\documentclass[11pt,leqno,twoside]{amsart}
\usepackage{amssymb,amsmath,amsthm,soul,color}
\usepackage{t1enc}
\usepackage[cp1250]{inputenc}
\usepackage{a4,indentfirst,latexsym}
\usepackage{graphics}
\usepackage{mathrsfs}
\usepackage{cite,enumitem,graphicx}
\usepackage[colorlinks=true,urlcolor=blue,
citecolor=red,linkcolor=blue,linktocpage,pdfpagelabels,
bookmarksnumbered,bookmarksopen]{hyperref}
\usepackage[english]{babel}
\usepackage[left=2.61cm,right=2.61cm,top=2.72cm,bottom=2.72cm]{geometry}
\usepackage[metapost]{mfpic}
\usepackage[hyperpageref]{backref}
\usepackage[normalem]{ulem}
\usepackage{todonotes}

\newtheorem{Th}{Theorem}[section]
\newtheorem{Prop}[Th]{Proposition}
\newtheorem{Lem}[Th]{Lemma}
\newtheorem{Cor}[Th]{Corollary}

\theoremstyle{remark}
\newtheorem{Rem}[Th]{Remark}
\theoremstyle{definition}

\newcommand{\vp}{\varphi}

\newcommand{\eps}{\varepsilon}

\def\N{\mathbb{N}}
\def\R{\mathbb{R}}

\newcommand{\cC}{{\mathcal C}}
\newcommand{\cD}{{\mathcal D}}

\newcommand{\cG}{{\mathcal G}}
\newcommand{\cH}{{\mathcal H}}

\newcommand{\cM}{{\mathcal M}}

\newcommand{\cO}{{\mathcal O}}

\newcommand{\cS}{{\mathcal S}}

\newcommand{\cX}{{\mathcal X}}

\newcommand{\la}{\lambda}
\newcommand{\De}{\Delta}

\newcommand{\weakto}{\rightharpoonup}

\numberwithin{equation}{section}

\begin{document}
	\title{Normalized solutions to Schr\"odinger equations in the strongly sublinear regime}
\author[J. Mederski]{Jaros\l aw Mederski}
\author[J. Schino]{Jacopo Schino}

%\address[J. Mederski]{\newline\indent
%	{\newline\indent 
%		Institute of Mathematics,
%		\newline\indent 
%		Polish Academy of Sciences,
%		\newline\indent 
%		ul. \'Sniadeckich 8, 00-656 Warsaw, Poland}
%}
%\email{\href{mailto:jmederski@impan.pl}{jmederski@impan.pl}}
\address[J. Mederski]{\newline\indent
	Institute of Mathematics,
	\newline\indent 
	Polish Academy of Sciences,
	\newline\indent 
	ul. \'Sniadeckich 8, 00-656 Warsaw, Poland
}
\email{\href{mailto:jmederski@impan.pl}{jmederski@impan.pl}}

\address[J. Schino]{\newline\indent
	Department of Mathematics
	\newline\indent 
	North Carolina State University
	\newline\indent 
	2311 Stinson Drive, 27607 Raleigh, NC, USA
	\newline\indent
	and
	\newline\indent
	Faculty of Mathematics, Informatics and Mechanics
	\newline\indent
	University of Warsaw
	\newline\indent
	ul. Banacha 2, 02-097, Warsaw, Poland
}
\email{\href{mailto:j.schino2@uw.edu.pl}{j.schino2@uw.edu.pl}}

\maketitle

\begin{abstract}
We look for solutions to the Schr\"odinger equation
\[
-\Delta u + \lambda u = g(u) \quad \text{in } \R^N
\]
coupled with the mass constraint $\int_{\R^N}|u|^2\,dx = \rho^2$, with $N\ge2$. The behaviour of $g$ at the origin is allowed to be strongly sublinear, i.e., $\lim_{s\to0}g(s)/s = -\infty$, which includes the case
\[
g(s) = \alpha s \ln s^2 + \mu |s|^{p-2} s
\]
with $\alpha > 0$ and $\mu \in \R$, $2 < p \le 2^*$ properly chosen. We consider a family of approximating problems that can be set in $H^1(\R^N)$ and the corresponding least-energy solutions, then we prove that such a family of solutions converges to a least-energy one to the original problem. Additionally, under certain assumptions about $g$ that allow us to work in a suitable subspace of $H^1(\R^N)$, we prove the existence of infinitely many solutions.
\end{abstract}


%{\bf Key words:} Nonlinear scalar field equations, logarithmic Sobolev inequality, cubic-quintic effect, critical point theory, nonradial solutions, concentration compactness, Lions' lemma, Pohozaev manifold, zero mass case, infinite mas case.
%

\section*{Introduction}
\setcounter{section}{1}

We look for solutions $(u,\lambda) \in H^1(\R^N) \times \R$ to the following equation
\begin{equation}\label{eq}
-\Delta u + \lambda u = g(u) \quad \mathrm{in} \ \R^N, \ N \geq 2
\end{equation}
paired with the constraint
\begin{equation}\label{eq:mass}
\int_{\R^N} |u|^2 \, dx =\rho^2,
\end{equation}
where $\rho>0$ is fixed and $g$ satisfies suitable assumptions (see (g0)--(g4) below). Equations like \eqref{eq} arise when one looks for standing-wave solutions to the evolution equation
\begin{equation}\label{eq:time}
\mathrm{i} \partial_t \Psi + \Delta \Psi + g(\Psi) = 0,
\end{equation}
i.e., $\Psi(t,x) = e^{\mathrm{i} \lambda t} u(x)$. The condition \eqref{eq:mass}, instead, is motivated by the property that, formally, the quantity
\[
\int_{\R^N} |\Psi(t,x)|^2 \, dx
\]
is conserved in time. Moreover, the $L^2$-norm squared of a solution to \eqref{eq:time} has a precise physical meaning in the contexts which this equation is derived from: it represents, e.g., the power supply in nonlinear optics and the total number of atoms in the Bose--Einstein condensation.
The model for $g$ that inspires this paper is
\begin{equation}\label{eq:BBM}
g(s) = s \ln s^2,
\end{equation}
which was introduced in \cite{Bialynicki} to obtain the {\em separability of non-interacting subsystems}; observe that, in this case,
\begin{equation*}
\lim_{s \to 0} \frac{g(s)}{s} = -\infty,
\end{equation*}
which is known in the literature as the {\it strongly sublinear} (or {\it infinite-mass}) regime. After that, both \eqref{eq} -- possibly with an external potential -- and \eqref{eq:time}, in both cases with $g$ as in \eqref{eq:BBM}, raised great interest: we mention \cite{dAvenia,SqSz,TZ,WangZhang,WangZhangJMPA} for the former context and \cite{CarlesGallagher,CazeLog,CazHar,GLN} (see also \cite[Section 9.3]{Cazenave:book}) for the latter one. The common feature of most of the articles above is that the authors can work in the subspace of $H^1(\R^N)$
\begin{equation*}%\label{eq:H1log}
\cH := \left\{u \in H^1(\R^N) : \int_{\R^N} u^2 |\ln u^2| \, dx < \infty\right\},
\end{equation*}
where the energy functional (see \eqref{eq:action} below) is of class $\cC^1$ (we mention that, instead, the approach of \cite{dAvenia} is based on critical point theory for lower semi-continuous functionals, while that of \cite{GLN} on a truncation argument). Of course, this strategy is impossible with generic nonlinearities as in this paper; speaking of which, generic nonlinear terms were recently considered in \cite{ITWZ,MederskiBL2}.

Concerning normalized solutions, i.e., in the presence of \eqref{eq:mass}, to the best of our knowledge, they are only treated in two very recent pieces of work: in \cite{Shuai_Yang}, with
\begin{equation}\label{eq:log+p}
g(s) = \alpha s \ln s^2 + \mu |s|^{p-2} s,
\end{equation}
%\[
%\frac{\alpha}{2} s^2 \left(\ln s^2 - 1\right) + \frac{\mu}{p} |s|^p
%\]
where $\alpha,\mu \in \R$ and $2 < p \le 2^*$, and in \cite{ZhangZou}, which mainly concerns systems of two equations. It is worth mentioning that, when $g$ is as in \eqref{eq:BBM}, if $(v,0)$ solves \eqref{eq} and $v \ne 0$, then $(u,\lambda)$ given by
\[
u = \frac{\rho}{|v|_2} v, \quad \lambda = 2 \ln\left(\frac{\rho}{|v|_2}\right)
\]
is a solution to \eqref{eq}--\eqref{eq:mass}. Nonetheless, the approach of \cite{Shuai_Yang,ZhangZou} still makes use of $\cH$. In particular, to recover compactness, the authors consider the subspace of radially symmetric functions as in \cite{CazeLog,CazHar}, which embeds compactly into $L^q(\R^N)$ for every $q \in [2,2^*)$; this is different from the classical result concerning the subspace of radial functions in $H^1(\R^N)$, which embeds compactly into $L^q(\R^N)$ only for $q \in (2,2^*)$. The additional compact embedding into $L^2(\R^N)$ proves to be extremely useful owing to the constraint \eqref{eq:mass}; in particular, this is what allows the authors to work with every $\rho>0$ when $2 < p < 2 + 4/N$, where $p$ is the same as in \eqref{eq:log+p}. Once again, here we cannot utilize the compact embedding into $L^2(\R^N)$ because $g$ need not be given by \eqref{eq:log+p}.

The function $g$ in \eqref{eq} satisfies the following assumptions. Define $G(s)=\int_0^s g(t)\, dt$ and
\begin{equation}\label{eq:G+}
G_+(s) =
\begin{cases}
\displaystyle \int_0^s \max\{g(t),0\} \, dt & \text{if } s \ge 0,\\
\displaystyle \int_{s}^0 \max\{-g(t),0\} \, dt & \text{if } s < 0.
\end{cases}
\end{equation}
\begin{itemize}
	\item[(g0)] $g\colon\R \to \R$ is continuous and $g(0)=0$.
	\item[(g1)] $\lim_{s\to 0}G_+(s)/|s|^2=0$ and $\limsup_{s \to 0} g_+(s)/s < \infty$, where $g_+ = G_+'$.
	\item[(g2)] If $N\ge3$, then $\limsup_{|s|\to \infty}|g(s)|/|s|^{2^*-1}<\infty$, where $2^*=\frac{2N}{N-2}$.\\
	If $N=2$, then $\lim_{|s|\to \infty}g(s)/e^{\alpha s^2} = 0$ for all $\alpha > 4\pi$.
	\item[(g3)] $\lim_{|s|\to \infty} G_+(s) / |s|^{2+4/N}=0$.
	\item[(g4)] There exists  $\xi_0 \ne 0$ such that $G(\xi_0)>0$.
\end{itemize}
The assumptions (g0), (g2), and (g4) are classical in the context of elliptic problems (cf. \cite{BerLions,BerLionsII}), while (g1) is a weaker version of what is normally assumed in the presence of $L^2$-constraints, i.e., $\lim_{s\to 0} g(s) / s = 0$; (g3) indicates that the growth of $G_+$ is mass-subcritical at infinity. With these assumptions, we can deal with nonlinearities that are much more general than \eqref{eq:log+p}; for example, we can consider functions $g$ that behave like $-|s|^{\omega - 1} s$ or $|s|^{\omega - 1} s \ln s^2$ for $|s|$ small, with $0 < \omega < 1$. We can also cover the case when $g$ goes to $0$ more slowly than any power, e.g., when $g$ behaves like $1 / \ln s^2$ for $|s|$ small.

Let us define
\[
\cD := \left\{ u \in H^1(\R^N) : \int_{\R^N} |u|^2 \, dx \le \rho^2 \right\} \quad \text{and} \quad \cS := \left\{ u \in H^1(\R^N) : \int_{\R^N} |u|^2 \, dx = \rho^2 \right\}.
\]
Sometimes, we will write explicitly $\cD(\rho)$ and $\cS(\rho)$. Let also $J \colon H^1(\R^N) \to \R \cup \{\infty\}$ be the energy functional associated with \eqref{eq}--\eqref{eq:mass}, i.e.,
\begin{equation}\label{eq:action}
J(u) = \frac12 \int_{\R^N} |\nabla u|^2 \, dx - \int_{\R^N} G(u) \, dx.
\end{equation}

Our first result is the existence of a solution to \eqref{eq}--\eqref{eq:mass}, with additional properties. We point out that by \emph{solution} to \eqref{eq} we mean a pair $(u,\lambda) \in H^1(\R^N) \times \R$ such that $J'(u)v = -\lambda uv$ for every $v \in \cC^{\infty}_0(\R^N)$.
\begin{Th}\label{th:mainsub}
If $g$ satisfies (g0)--(g4), then there exists $\bar{\rho}>0$ such that for every $\rho>\bar{\rho}$ there exist $\lambda>0$ and $u\in\cS$ such that $J(u)=\min_{\cD}J<0$ and $(u,\lambda)$ is a solution to \eqref{eq}--\eqref{eq:mass}. Moreover, $u$ has constant sign and is, up to a translation, radial and radially monotonic.
\end{Th}

\begin{Rem}\label{rem:FiniteMass}
Of course, (g0)--(g4) include the classical case {\color{blue}$\lim_{s \to 0} g(s) / s = 0$}, which implies that $J$ is of class $\cC^1$ in $H^1(\R^N)$. In this case, with the notations of Theorem \ref{th:mainsub}, $u$ is actually a critical point of $J|_\cS$. In addition, if $\lim_{s \to 0} G(s) / |s|^{2+4/N} = \infty$ (i.e., in the mass-subcritical regime), then $\bar{\rho} = 0$ in Theorem \ref{th:mainsub} (see, e.g., \cite{Schino} for more details). An important example in this context is given by nonlinear media with \emph{saturation} \cite{Langbein}, where $g(s) = s^3 / (1 + s^2)$.
\end{Rem}

\begin{Rem}\label{rem:sign_lam}
Under the assumptions (g0)--(g4), if there exists a solution $(u,\lambda) \in H^1(\R^N) \times \R$ to \eqref{eq}--\eqref{eq:mass} with $J(u) \le 0$, then $\lambda > 0$. As a matter of fact, $J(u) \le 0$ implies that $\int_{\R^N} G(u) \, dx$ is finite, thus from \cite[Proposition 3.1]{MederskiBL2} $(u,\lambda)$ satisfies the Poho\v{z}aev identity $(N-2) \int_{\R^N} |\nabla u|^2 \, dx = 2N \int_{\R^N} G(u) - \frac12 \lambda |u|^2 \, dx$ and, consequently, $0 \ge J(u) = \frac1N |\nabla u|_2^2 - \frac{\lambda}{2} \rho^2$.
\end{Rem}

Our next result concerns the limit of the ground state energy map $\rho \mapsto \inf_{\cD(\rho)} J$ as $\rho \to \infty$.
\begin{Prop}\label{pr:gsem}
If $g$ satisfies (g0)--(g4), then
\[
\lim_{\rho \to \infty} \inf_{\cD(\rho)} J = -\infty.
\]
\end{Prop}

Moreover, we obtain the relative compactness in $H^1(\R^N)$ of minimizing sequences for $\inf_{\cD(\rho)} J$ up to translations. Here and in the sequel, with a small abuse of notations, if $(u_n)_n$ is a sequence in $X$, we write $u_n \in X$ instead of $(u_n)_n \subset X$.
\begin{Th}\label{th:relcom}
If (g0)--(g4) hold, then there exists $\bar{\rho} > 0$ such that for every $\rho > \bar{\rho}$, if $0 < \rho_n \to \rho$ and $u_n \in \cD(\rho_n)$ is such that $J(u_n) \to \inf_{\cD(\rho)} J$, then $u_n$ is relatively compact in $H^1(\R^N)$ up to translations.
\end{Th}It is well known that such a relative compactness is an important step toward the orbital stability of solutions to \eqref{eq}--\eqref{eq:mass}. Unfortunately, we currently do not have the necessary tools to investigate the dynamics of \eqref{eq:time} in this general setting, hence this issue is postponed to future work. Moreover, unlike the classical approach when $J$ is of class $\cC^1$, where the existence of a constrained minimizer for $J$ and, consequently, a solution to \eqref{eq}--\eqref{eq:mass} is a by-product of the aforementioned relative compactness, here we prove first the existence of a solution to \eqref{eq}--\eqref{eq:mass} and then use this fact to prove the compactness result.

When $g$ is as in \eqref{eq:log+p}, we additionally have the following non-existence result.
\begin{Th}\label{th:nonex}
Let $g$ be given by \eqref{eq:log+p} with $\alpha > 0$, $\mu \in \R$, and $2 < p \le 2^*$ ($p > 2$ if $N = 2$).

(i) If $\rho \gg 1$ and either $-\frac{\alpha p}{p-2}e^{-p/2} < \mu \le 0$, or $\mu > 0$ and $2<p<2+4/N$, then there exists a solution $(u,\lambda) \in H^1(\R^N) \times (0,\infty)$ to \eqref{eq}--\eqref{eq:mass} such that $G(u) \in L^1(\R^N)$ and $J(u) = \min_{\cD}J < 0$.

(ii) If $\mu \le -\frac{\alpha p}{p-2}e^{-p/2}$, then there exist no solutions $(u,\lambda) \in H^1(\R^N) \times [0,\infty)$ to \eqref{eq}--\eqref{eq:mass} such that $G(u) \in L^1(\R^N)$.
\end{Th}

\begin{Rem}\label{rem:+energy}
In Theorem \ref{th:mainsub}, we look for (ground state) solutions with sufficiently large $\rho$ because this implies that the energy is negative, which is used, in turn, to ensure such an existence. Then, as we pointed out in Remark \ref{rem:sign_lam}, a consequence is that $\lambda>0$. On the other hand, in the context of Theorem \ref{th:nonex} (i) and $\mu>0$, ground state solutions still exist for all $\rho>0$ as in Corollary \ref{cor:corlog} below, but with no information about the Lagrange multiplier or the energy.
%cf. \cite[Theorem 1.1]{ZhangZhang}, but with no information about the Lagrange multiplier and with additional assumptions about the nonlinear term.
\end{Rem}

We follow the approach of \cite{MederskiBL2} and consider a family of approximating problems. Note that the approximations are different from those in \cite{ITWZ}, where the nonlinearity is modified in a neighbourhood of infinity, and are more similar to those in \cite{GLN}, although not the same as the nonlinearity therein is of the form \eqref{eq:BBM}.

Let $g_{-}(s):=g_{+}(s)-g(s)$ and $G_{-}(s):=G_{+}(s)-G(s)\geq 0$ for $s\in\R$.
In view of (g1) and (g3), $G_{+}(u)\in L^{1}(\R^N)$ for $u \in H^1(\R^N)$; however,
$G_{-}(u)$ may not be integrable
unless $G_{-}(s) \lesssim |s|^{2}$ for small $|s|$. Moreover, for any $u \in H^1(\R^N)$, $J$ is differentiable at $u$ along functions in $\cC_0^\infty(\R^N)$, but it is not of class $\cC^1$ in general. In order to overcome this problem, for every $\eps\in (0,1)$ let us take an even function $\vp_\eps\colon\R\to [0,1]$ such that $\vp_\eps(s) = |s| / \varepsilon$ for $|s|\leq \eps$, $\vp_{\eps}(s) = 1$ for $|s|\geq \eps$. We introduce a new functional,
\begin{equation}\label{eq:actionEPS}
J_\eps(u) = \frac12 \int_{\R^N} |\nabla u|^2 \, dx + \int_{\R^N} G_{-}^\eps(u) \, dx - \int_{\R^N} G_{+}(u) \, dx,
\end{equation}
such that $G_{-}^\eps(s)=\int_0^s \vp_\eps(t)g_-(t)\, dt$, $s\in\R$,
and now observe that $G_{-}^\eps(s) \leq c_\eps|s|^{2}$ for every $|s| \le 1$ and some constant $c_\eps>0$ depending only on $\eps>0$.
Hence, for $\eps\in (0,1)$,  $J_\eps$ is of class $\cC^1$ on $H^1(\R^N)$.
%well-defined on $H^1(\R^N)$, continuous, and $J_\eps'(u)(v)$ exists for every $u,v \in H^1(\R^N)$.
%Hence we call $u$ a {\em critical point} of $J_\eps$ provided that $J_\eps'(u)(v)=0$ for any $v\in\cC_0^\infty(\R^N)$.

The idea of considering the $L^2$-disc $\cD$ instead of the $L^2$-sphere $\cS$ was introduced in \cite{BiegMed} in a context where the nonlinear term is mass-supercritical and Sobolev-subcritical at infinity (i.e., $\lim_{|s| \to \infty} G(s) / |s|^{2+4/N} = \infty$, $\lim_{|s| \to \infty} G(s) / |s|^{2^*} = 0$) and mass-critical or -supercritical at the origin (i.e., $\limsup_{s \to 0} G(s) / |s|^{2+4/N} < \infty$); later on, it was exploited again in \cite{MederskiSchino}, in a similar context that allows systems with Sobolev-critical nonlinearities, in \cite{Schino}, for equations and systems of equations in the mass-critical or -subcritical setting, and in \cite{BMS}, where polyharmonic operators and Hardy-type potentials are considered. One of the reasons for this choice was that the embedding $H^1(\R^N) \hookrightarrow L^2(\R^N)$ is not compact, even if considering radial functions or with symmetries as in \cite{Lions82}, which, in turn, causes that the limit point of a weakly convergent sequence in $\cS$ need not belong to $\cS$; on the contrary, limit points of weakly convergent sequences in $\cD$ stay in $\cD$, and the additional information that such a weak limit point belongs to the set one is working with proves to be useful. In this paper, this strategy is fundamental: it allows the critical point found in Theorem \ref{th:pert_n} below -- the existence result for the perturbed problem -- to be a minimizer of the corresponding energy functional over the whole set $\cD$, and this fact plays an important role in the proof of Theorem \ref{th:mainsub} because the ``candidate minimizer'' of $J|_\cD$ might not, initially, belong to $\cS$.

Now, we turn to the problem of finding multiple solutions (in fact, infinitely many) to \eqref{eq}--\eqref{eq:mass}. In this case, we need to change both the assumptions about $g$ and our approach. The idea is to build a subspace of $H^1(\R^N)$ in the spirit of $\cH$, but with a more general setting. The issues with using the previous approach to obtain an infinite sequence of solutions are mainly two. First, even at the level of the perturbed functional, one can obtain an arbitrary large number of solutions if $\rho$ is sufficiently large, but it is unclear whether infinitely many solutions exist; this is caused by the lack of compact embedding into $L^2(\R^N)$, which in turn raises the need for each of the usual minimax values (cf. Theorem \ref{Th:JL} below) to be negative, for which we need $\rho$ to be large (and how large depends differently on each minimax level). Second, fixed a family (depending on $\varepsilon$) of critical values of the perturbed functional obtained in the way we have just briefly mentioned, and considering the corresponding family of critical points, it is not clear whether this family converges to a critical point of the unperturbed functional, unlike what happens in the Proof of Theorem \ref{th:mainsub}, where we can exploit the information that the critical levels are minima.

We assume that the right-hand side in \eqref{eq} is given by
\begin{equation}\label{eq:gfa}
g(s) = f(s) - a(s)
\end{equation}
together with the following assumptions, where $F(s) = \int_0^s f(t) \, dt$, $A(s) = \int_0^s a(t) \, dt$, $F_+$ is defined via \eqref{eq:G+} replacing $g$ with $f$, and $f_+ = F_+'$.
\begin{itemize}
	\item [(A)] $A \in \cC^1(\R)$ is an \textit{N}-function that satisfies the $\De_2$ and $\nabla_2$ conditions globally and such that $\lim_{s \to 0} A(s)/s^2 = \infty$ and $s \mapsto a(s)s$ is convex.
	\item [(f0)] $f \colon \R \to \R$ is continuous and odd.
	\item [(f1)] $\lim_{s \to 0} f_+(s) / s = 0$.
	%\item [(f4)] There exists $\xi_0 > 0$ such that $F(\xi_0) > A(\xi_0)$.
	\item [(f2)] If $N \ge 3$, then $|f(s)| \lesssim a(s) + |s| + |s|^{2^*-1}$;\\
	If $N = 2$, then for all $q \ge 2$ and $\alpha > 4\pi$ there holds $|f(s)| \lesssim a(s) + |s| + |s|^{q-1} \left(e^{\alpha s^2} - 1\right)$.
	\item [(f3)] $\limsup_{|s| \to \infty} f_+(s) / |s|^{1+4/N} < \infty$.
\end{itemize}
Observe that (f3) implies
\begin{equation}\label{eq:eta}
\eta := \limsup_{|s| \to \infty} \frac{F_+(s)}{|s|^{2+4/N}} < \infty.
\end{equation}

Let us define
\begin{equation*}
W := \left\{u \in H^1(\R^N) : A(u) \in L^1(\R^N) \right\}\end{equation*}
and, for a subgroup $\cO \subset \cO(N)$,
\begin{equation*}
W_{\cO} := \left\{u \in W : u(g\cdot) = u \text{ for all } g \in \cO\right\}.
\end{equation*}
If $N = 4$ or $N \ge 6$, with the purpose of finding non-radial solutions (cf. \cite{BartschWillem,MederskiBL}), fix $2 \le M \le N/2$ such that $N - 2M \ne 1$ and let $\tau(x) = (x_2,x_1,x_3)$ for $x = (x_1,x_2,x_3) \in \R^M \times \R^M \times \R^{N-2M}$ ($x_3$ is removed form the definition of $\tau$ if $N = 2M$). Then, let
\begin{equation*}
\cX := \left\{u \in W_{\cO(M) \times \cO(M) \times \cO(N-2M)} : u \circ \tau = -u\right\}
\end{equation*}
($\cO(N-2M)$ is removed from the definition of $\cX$ if $N = 2M$) and observe that $\cX \cap W_{\cO(N)} = \{0\}$.

For $2 < p < 2^*$, let us denote by $C_{N,p}$ the best constant in the Gagliardo--Nirenberg inequality:
\begin{equation}\label{eq:GN}
|u|_p \le C_{N,p} |\nabla u|_2^{N(1/2-1/p)} |u|_2^{1-N(1/2-1/p)} \quad \forall \, u \in H^1(\R^N).
\end{equation}

Our multiplicity result reads as follows.
\begin{Th}\label{th:multip}
Let $g$ be given by \eqref{eq:gfa} such that (A), (f0)--(f3), and
\begin{equation}\label{eq:eta-rho}
2 \eta C_{N,2+4/N}^{2+4/N} \rho^{4/N} < 1
\end{equation}
hold, where $\eta$ is defined in \eqref{eq:eta}.
%Let also $\cO \subset \cO(N)$ be a subgroup compatible with $\R^N$.
Then there exist infinitely many solutions $(u,\lambda) \in W_{\cO(N)} \times \R$ to \eqref{eq}--\eqref{eq:mass}, one of which -- say, $(\bar{u},\bar{\lambda})$ -- having the property that $J(\bar u) = \min_{\cS \cap W} J$. If, in addition, $N = 4$ or $N \ge 6$, then there exist infinitely many solutions $(v,\mu) \in \cX \times \R$, one of which -- say, $(\bar{v},\bar{\mu})$ -- having the property that $J(\bar{v}) = \min_{\cS \cap \cX} J$.
\end{Th}

By setting
\begin{equation*}
A(s) :=
\begin{cases}
- \alpha s^2 \ln s^2 & \text{if } 0 \le |s| < e^{-3},\\
3 \alpha s^2 + 4 \alpha e^{-3}|s| - \alpha e^{-6} & \text{if } |s| \ge e^{-3},
\end{cases}
\qquad
F(s) = \alpha s^2 \ln s^2 + A(s) + \frac{\mu}{p} |s|^p,
\end{equation*}
with $\alpha > 0$, and $\mu,p$ in the right range, we check that (A) and (f0)--(f3) hold and we recover the case \eqref{eq:log+p}.

\begin{Cor}\label{cor:corlog}
Let $g$ be given by \eqref{eq:log+p} with $\alpha > 0$, $\mu \geq 0$, and $2 < p \le 2+4/N$.
If \eqref{eq:eta-rho} holds, then there exists a solution $(u,\lambda) \in W\times \R$ to \eqref{eq}--\eqref{eq:mass} such that $J(u) = \min_{\cS\cap W}J$. 
\end{Cor}

In particular, we get a solution to \eqref{eq}--\eqref{eq:mass} with \eqref{eq:log+p}  for any $\rho>0$ provided that  $2<p<2+4/N$.
Moreover, (f3) allows us to handle terms $g$ that grow ``too fast'' to be included in the usual $H^1(\R^N)$ setting. For example, letting
\begin{equation*}
A(s) :=
\begin{cases}
\displaystyle -\alpha s^2 \ln s^2 & \text{if } 0 \le |s| < e^{-3},\\
\displaystyle \frac{10 \alpha}{p} e^{3p-6}|s|^p + 2 \alpha \left(3 - \frac5p\right) e^{-6} & \text{if } |s| \ge e^{-3},
\end{cases}
\end{equation*}
with $\alpha > 0$, we can allow $f_-(s) \approx |s|^{p-2} s$ for $|s| \gg 1$ for every $p > 2$, even when $N \ge 3$ and $p > 2^*$. At the same time, we can treat nonlinearities that are ``more singular'' at the origin, for example, letting
\begin{equation*}
A(s) := \frac1q |s|^q, \quad 1 < q < 2,
\end{equation*}
and taking any $f$ that satisfies (f0)--(f3). The price to pay for this different approach is that we need \eqref{eq:gfa} and (A) to construct a proper functional setting (i.e., an Orlicz one); this prevents us from dealing, e.g., with terms that behave like $1/\ln s^2$ near the origin, which instead are allowed by the use of the approximating problems to find least-energy solutions. On the other hand, we can weaken the assumptions about $f$, the ``remaining'' part of the nonlinearity: it can be controlled by $a$ instead of the usual conditions with an $H^1(\R^N)$-setting (in particular, we can deal with Sobolev-supercritical nonlinearities), we do not need $F > A$ anywhere, $f_+$ can have a mass-critical growth at infinity (provided $\rho$ is small), and there are no restrictions on $\rho$ if $f_+$ has a mass-subcritical growth at infinity; none of these scenarios can be handled with the previous approach. Except for \cite[Theorems 1.2 and 1.3]{Shuai_Yang}, where the nonlinear term has the very specific form \eqref{eq:log+p}, Theorem \ref{th:multip} seems to be the first multiplicity result when normalized solutions are sought in the strongly sublinear regime, and also the first time at all that more than two solutions are obtained.

The paper is structured al follows: in Section \ref{SecP}, we obtain least-energy solutions to the perturbed problem; in Section \ref{SecU}, we prove Theorems \ref{th:mainsub}, \ref{th:relcom}, \ref{th:nonex}, and Proposition \ref{pr:gsem}; in Section \ref{SecM}, we obtain infinitely many solutions to \eqref{eq}--\eqref{eq:mass}.

\section{Solutions to the perturbed problem}\label{SecP}

%\subsection{Ground states}
Recall from \eqref{eq:actionEPS} the definition of $J_\varepsilon \colon H^1(\R^N) \to \R$. The main purpose of this section is to prove the following result, where we denote $c_\varepsilon(\rho) := \inf_{\cD(\rho)} J_\varepsilon$.
\begin{Th}\label{th:pert_n}
Let $g$ satisfy (g0)--(g4). Then for every $\beta > 0$ there exists $\bar\rho > 0$ such that for every $\rho > \bar\rho$ and $\varepsilon > 0$ there exist $\lambda_\varepsilon > 0$ and $u_\varepsilon \in H^1(\R^N)$ such that $J_\varepsilon(u_\varepsilon) = c_\varepsilon(\rho) < -\beta$ and
\begin{equation}\label{eq:pert}
\begin{cases}
-\Delta u_\varepsilon + \lambda_\varepsilon u_\varepsilon = g_+(u_\eps) - \varphi_\eps(u_\eps) g_-(u_\varepsilon) \quad \text{in } \R^N\\
\int_{\R^N}|u_\varepsilon|^2\,dx=\rho^2.
\end{cases}
\end{equation}
Moreover, every such $u_\eps$ has constant sign and, up to a translation, is radial and radially monotonic.
\end{Th}

We begin by proving a series of lemmas.

\begin{Lem}\label{lem:coerbb}
If $g$ satisfies (g0)--(g3),
%\begin{equation*}%\label{eq:sub}
%\eta:=\limsup_{|s|\to \infty}\frac{G_+(s)}{|s|^{2+4/N}}
%\end{equation*}
%is such that
%\begin{equation}\label{eq:etainfsub}
%2C_{N,2_\#}^{2_\#}\eta\rho^{4/N}<1,
%\end{equation}
then $J_\varepsilon|_\cD$ is coercive and bounded below.
\end{Lem}
\begin{proof}
From (g1) and (g3), for every $\delta > 0$ there exists $C_\delta > 0$ such that for every $s \in \R$
\[
G_+(s) \le C_\delta |s|^2 + \delta |s|^{2+4/N}.
\]
Consequently, from \eqref{eq:GN}, for every $u \in \cD$ there holds
\[
\begin{split}
J_\varepsilon(u) & \ge \int_{\R^N} \frac12 |\nabla u|^2 - G_+(u) \, dx \ge \frac12 |\nabla u|_2^2 - C_\delta |u|_2^2 - \delta |u|_{2+4/N}^{2+4/N}\\
& \ge \frac12 |\nabla u|_2^2 - C_\delta \rho^2 - \delta C_{N,2+4/N}^{2+4/N} \rho^{4/N} |\nabla u|_2^2.
\end{split}
\]
Taking $\delta < \left(2 C_{N,2+4/N}^{2+4/N} \rho^{4/N}\right)^{-1}$ we conclude.
\end{proof}

%The following lemmas are inspired by \cite{BerLionsII} and \cite{JeanjeanLu:norm}.
%Note that, for every $\varepsilon\in(0,1)$, $G_-^\varepsilon\le G_-$, which implies
%\begin{equation}\label{eq:epsestim}
%G^\varepsilon\ge G \quad \text{and} \quad |G^\varepsilon|\le|G|.
%\end{equation}
%In fact, for every fixed $s\in\R$, $\varepsilon\mapsto G_-^\varepsilon(s)$ is non-increasing, which in turn implies
%\[
%\varepsilon<\varepsilon'\Rightarrow J\ge J_\varepsilon\ge J_{\varepsilon'}.
%\]
%For $\varepsilon,\rho>0$, let us denote $c_\varepsilon(\rho):=\inf_{\cD(\rho)}J_\varepsilon$ and $c(\rho) = \inf_{\cD(\rho)}J$.
%Moreover, for $s>0$ and $u \in H^1(\R^N)$ define $s\star u := s^{N/2}u(s\cdot)$. Note that $|s\star u|_2 = |u|_2$.

\begin{Lem}\label{lem:neg}
If $g$ satisfies (g0)--(g2) and (g4), then for every $\beta>0$ there exists $\bar{\rho}>0$ such that
\[
c_\eps(\rho) < -\beta
\]
for every $\rho>\bar{\rho}$ and every $\varepsilon\in(0,1)$.
\end{Lem}
\begin{proof}
Arguing as in \cite[Proof of Theorem 2]{BerLions} and observing that $G^\varepsilon := G_+ - G_-^\varepsilon \ge G$, there exists $u \in H^1(\R^N)$ such that $\int_{\R^N} G^\eps(u) \, dx \ge 1$ for every $\eps \in (0,1)$. Define $u_\rho := u(|u|_2^{2/N}\rho^{-2/N}\cdot) \in \cS(\rho)$. There holds
\[\begin{split}
J_\eps(u_\rho) & = \int_{\R^N} \frac12 |\nabla u_\rho|^2 - G^\eps(u_\rho) \, dx = \frac{|\nabla u|_2^2}{2|u|_2^{2(N-2)}} \rho^{2 - 4/N} - \frac{1}{|u|_2^2} \int_{\R^N} G^\eps(u) \, dx \rho^2\\
& \le \frac{|\nabla u|_2^2}{2|u|_2^{2(N-2)}} \rho^{2 - 4/N} - \frac{1}{|u|_2^2} \rho^2 \to -\infty
\end{split}\]
as $\rho \to \infty$, whence the statement.
\end{proof}

\begin{Lem}\label{lem:subadd}
If (g0)--(g3) hold, then for every $\rho_1,\rho_2>0$ there holds $c_\varepsilon(\sqrt{\rho_1^2+\rho_2^2})\le c_\varepsilon(\rho_1)+c_\varepsilon(\rho_2)$.
\end{Lem}
\begin{proof}
Fix $\delta>0$. From the definition of $c_\varepsilon$, which is finite due to Lemma \ref{lem:coerbb}, there exist $u_1,u_2\in\cC_0^\infty(\R^N)$ such that $|u_i|_2\le\rho_i$ and $J_\varepsilon(u_i)\le c_\varepsilon(\rho_i)+\delta$, $i=1,2$. In virtue of the translation invariance, we can assume that $u_1$ and $u_2$ have disjoint supports. Then $|u_1+u_2|_2^2 = |u_1|_2^2 + |u_2|_2^2 \le \rho_1^2 + \rho_2^2$ and so
\[
c_\varepsilon\left(\sqrt{\rho_1^2+\rho_2^2}\right) \le J_\varepsilon(u_1+u_2)\le c_\varepsilon(\rho_1)+c_\varepsilon(\rho_2)+2\delta.\qedhere
\]
\end{proof}

The following lemma is inspired from \cite[Lemma 3.2 (ii)]{Shibata_2014}.

\begin{Lem}\label{lem:strict}
Assume (g0)--(g3) hold and let $\rho_1,\rho_2 > 0$. If there exist $u_i\in\cD(\rho_i)$ such that $J_\varepsilon(u_i) = c_\varepsilon(\rho_i)$,  $i=1,2$, and $(u_1,u_2) \ne (0,0)$, then $c_\eps(\sqrt{\rho_1^2+\rho_2^2}) < c_\varepsilon(\rho_1) + c_\varepsilon(\rho_2)$.
\end{Lem}
\begin{proof}
For every $s\ge1$ and $i=1,2$ there holds
\[
c_\varepsilon(\sqrt{s} \rho_i) \le J_\varepsilon\bigl(u_i(s^{-1/N}\cdot)\bigr)=s\left(\frac1{2s^{2/N}}\int_{\R^N}|\nabla u_i|^2\,dx-\int_{\R^N}G^\varepsilon(u_i)\,dx\right)\le sJ_\varepsilon(u_i)=sc_\varepsilon(\rho_i).
\]
Moreover, if $s>1$ and $u_i \ne 0$, then $c_\varepsilon(\sqrt{s} \rho_i)<sc_\varepsilon(\rho_i)$. Let us assume, without loss of generality, that $u_1 \ne 0$. If $\rho_1\ge\rho_2$, then
\[
c_\varepsilon\left(\sqrt{\rho_1^2+\rho_2^2}\right) < \frac{\rho_1^2+\rho_2^2}{\rho_1^2} c_\varepsilon(\rho_1) = c_\varepsilon(\rho_1) + \frac{\rho_2^2}{\rho_1^2} c_\varepsilon(\rho_1) \le c_\varepsilon(\rho_1) + c_\varepsilon(\rho_2).
\]
If $\rho_1<\rho_2$, then
\[
c_\varepsilon\left(\sqrt{\rho_1^2+\rho_2^2}\right) \le \frac{\rho_1^2+\rho_2^2}{\rho_2^2} c_\varepsilon(\rho_2) = \frac{\rho_1^2}{\rho_2^2} c_\varepsilon(\rho_2) + c_\varepsilon(\rho_2) < c_\varepsilon(\rho_1) + c_\varepsilon(\rho_2).\qedhere
\]
\end{proof}

\begin{Rem}\label{rem:subadd}
(a) Lemmas \ref{lem:subadd} and \ref{lem:strict} still hold if $J_\eps$ is replaced with $J$.\\
(b) As proved in \cite[Lemma 2.5 (ii)]{Schino}, Lemma \ref{lem:strict} holds under the weaker assumption that $J_\eps(u_1) = c_\eps(\rho_1)$ or $J_\eps(u_2) = c_\eps(\rho_2)$.
\end{Rem}

\begin{Lem}\label{lem:relcom}
Let $g$ satisfy (g0)--(g4). Then, for every $\beta>0$, there exists $\bar\rho>0$ such that for every $\rho>\bar\rho$ and $\varepsilon>0$ we have $c_\varepsilon(\rho) < -\beta$ and, if $0 < \rho_n \to \rho$ and $u_n \in \cD(\rho_n)$ is such that $J_\eps(u_n) \to c_\eps(\rho)$, then there exists $u \in \cD(\rho) \setminus \{0\}$ such that, up to translations, $u_n \to u$ in $L^p(\R^N)$, $2<p<2^*$.
\end{Lem}
\begin{proof}
Let $\bar\rho>0$ be determined by Lemma \ref{lem:neg} and consider $u_n$ and $\rho_n$ as in the statement. Then $u_n$ is bounded in view of Lemma \ref{lem:coerbb}. If
\[
\lim_n\max_{y\in\R^N}\int_{B_1(y)}|u_n|^2\,dx=0,
\]
then $\lim_n|u_n|_p=0$ for every $2<p<2^*$ due to Lions' lemma \cite[Lemma I.1 in Part 2]{Lions84}, hence, taking into account (g1) and (g3), $\lim_n\int_{\R^N}G_+(u_n)\,dx=0$ and, consequently, $c_\varepsilon(\rho)=\lim_nJ_\varepsilon(u_n)\ge0$, a contradiction with Lemma \ref{lem:neg}. Therefore, there exist $y_n\in\R^N$ and $u\in\cD(\rho)\setminus\{0\}$ such that, up to a subsequence, $u_n(\cdot-y_n)\rightharpoonup u$ in $H^1(\R^N)$ and $u_n(\cdot-y_n)\to u$ a.e. in $\R^N$. Let us denote $v_n(x):=u_n(x-y_n)-u(x)$. In virtue of the Brezis--Lieb lemma \cite[Theorem 2]{BrezisLieb},
\[
\lim_nJ_\varepsilon(u_n)-J_\varepsilon(v_n)=J_\varepsilon(u).
\]
We want to prove that
\[
\lim_n\max_{y\in\R^N}\int_{B_1(y)}|v_n|^2\,dx=0,
\]
which, as before, will yield the statement of the Lemma. If this is not the case, then, as before, there exist $z_n\in\R^N$ and $v\in H^1(\R^N)\setminus\{0\}$ such that, denoting $w_n(x):=v_n(x-z_n)-v(x)$, we have $w_n\rightharpoonup0$ in $H^1(\R^N)$, $w_n\to0$ a.e. in $\R^N$, and
\[
\lim_nJ_\varepsilon(v_n)-J_\varepsilon(w_n)=J_\varepsilon(v).
\]
Observe that, again owing to the Brezis--Lieb lemma,
\[
\lim_n |u_n|_2^2 - |w_n|_2^2 = \lim_n |u_n|_2^2 - |v_n|_2^2 + |v_n|_2^2 - |w_n|_2^2 = |u|_2^2 + |v|_2^2,
\]
whence, denoting $\beta:=|u|_2>0$ and $\gamma:=|v|_2>0$, there holds
\[
\rho^2 - \beta^2 - \gamma^2 \ge \limsup_n |u_n|_2^2 - \beta^2 - \gamma^2 = \limsup_n |w_n|_2^2 =: \delta^2 \ge 0.
\]
If $\delta>0$, then let us set $\tilde{w}_n:=\frac{\delta}{|w_n|_2}w_n\in\cS(\delta)$. Via explicit computations (cf., e.g., \cite[Lemma 2.4]{Shibata_2014}), we have $\lim_nJ_\varepsilon(w_n)-J_\varepsilon(\tilde{w}_n)=0$, hence, together with Lemma \ref{lem:subadd} and since $c_\varepsilon$ is non-increasing,
\begin{equation}\label{eq:del}\begin{split}
c_\varepsilon(\rho) & = \lim_n J_\varepsilon(u_n) = J_\varepsilon(u) + J_\varepsilon(v) + \lim_n J_\varepsilon(w_n) = J_\varepsilon(u) + J_\varepsilon(v) + \lim_n J_\varepsilon(\tilde{w}_n)\\
& \ge c_\varepsilon(\beta) + c_\varepsilon(\gamma) + c_\varepsilon(\delta) \ge c_\varepsilon(\sqrt{\beta^2+\gamma^2+\delta^2}) \ge c_\varepsilon(\rho).
\end{split}\end{equation}
Thus all the inequalities in \eqref{eq:del} are in fact equalities and, in particular, $J_\varepsilon(u) = c_\varepsilon(\beta)$ and $J_\varepsilon(v) = c_\varepsilon(\gamma)$. Therefore Lemma \ref{lem:strict} yields $c_\varepsilon(\beta) + c_\varepsilon(\gamma) + c_\varepsilon(\delta) > c_\varepsilon(\sqrt{\beta^2+\gamma^2+\delta^2})$, which contradicts \eqref{eq:del}.
If $\delta=0$, then $w_n\to0$ in $L^2(\R^N)$, which, together with \eqref{eq:GN}, implies $\lim_n \int_{\R^N} G_+(w_n) \, dx = 0$, whence $\liminf_n J_\varepsilon(w_n) \ge 0$. Then \eqref{eq:del} becomes
\[
c_\varepsilon(\rho) = \lim_n J_\varepsilon(u_n) = J_\varepsilon(u) + J_\varepsilon(v) + \lim_n J_\varepsilon(w_n) \ge c_\varepsilon(\beta) + c_\varepsilon(\gamma) \ge c_\varepsilon(\sqrt{\beta^2+\gamma^2}) = c_\varepsilon(\rho)
\]
and we reach a contradiction as before.
\end{proof}

Let us recall that every solution $(u_\varepsilon,\lambda_\varepsilon)$ to the differential equation in \eqref{eq:pert} satisfies the Poho\v{z}aev identity
\begin{equation}\label{eq:Pohopert}
\int_{\R^N} (N-2) |\nabla u_\varepsilon|^2 + N \lambda_\varepsilon |u_\varepsilon|^2 \, dx = 2N \int_{\R^N} G_+(u_\eps) - G_-^\varepsilon(u_\varepsilon) \, dx.
\end{equation}

\begin{proof}[Proof of Theorem \ref{th:pert_n}]
Let $\bar\rho>0$ be determined by Lemma \ref{lem:neg} and consider a minimizing sequence $u_n\in\cD(\rho)$ for $c_\varepsilon(\rho)$. Then, in virtue of Lemmas \ref{lem:coerbb} and \ref{lem:relcom}, there exists $u_\varepsilon\in\cD(\rho)\setminus \{0\}$ such that, up to subsequences and translations, $u_n\rightharpoonup u_\varepsilon$ in $H^1(\R^N)$ and $u_n\to u_\varepsilon$ in $L^p(\R^N)$ for every $2<p<2^*$ and a.e. in $\R^N$. From (g1) and (g3) we have $\lim_n\int_{\R^N}G_+(u_n)\,dx = \int_{\R^N} G_+(u_\eps) \, dx$, therefore, using Fatou's Lemma, $c_\varepsilon(\rho)\le J_\varepsilon(u_\varepsilon)\le\lim_nJ_\varepsilon(u_n)=c_\varepsilon(\rho)<-\beta$. In particular, there exists $\lambda_\varepsilon\in\R$ such that
\[
-\Delta u_\varepsilon+\lambda_\varepsilon u_\varepsilon=g^\varepsilon(u_\varepsilon) \quad \text{in } \R^N.
\]
Note that $\lambda_\varepsilon=0$ if $u_\varepsilon$ is an interior point of $\cD(\rho)$, i.e., $|u_\varepsilon|_2<\rho$. If $\lambda_\varepsilon\le0$, then \eqref{eq:Pohopert} yields $0>J_\varepsilon(u_\varepsilon)\ge|\nabla u_\varepsilon|_2^2/N\ge0$, a contradiction, therefore $\lambda_\varepsilon>0$ and $u_\varepsilon\in\cS(\rho)$. The last statement follows from \cite[Theorem 1.4]{JL_min} (see also \cite{Maris}).
\end{proof}

\begin{Rem}\label{rem:rcL2}
The proof of Theorem \ref{th:pert_n} shows that every minimizing sequence for $c_\varepsilon(\rho)$ is relatively compact in $H^1(\R^N)$ up to a translation. As a matter of fact, from $\lim_n\int_{\R^N}G_+(u_n)\,dx=\int_{\R^N}G_+(u_\varepsilon)\,dx$ and $\lim_nJ_\varepsilon(u_n)=J_\varepsilon(u_\varepsilon)$ we get $\lim_n|\nabla u_n|_2=|\nabla u_\varepsilon|_2$. Moreover,
\[
\rho=|u_\varepsilon|_2\le\liminf_n|u_n|_2\le\limsup_n|u_n|_2\le\rho,
\]
i.e., $\lim_n|u_n|_2=|u_\varepsilon|_2$. Recalling that $u_n\rightharpoonup u_\varepsilon$ in $H^1(\R^N)$, we obtain $u_n\to u_\varepsilon$ in $H^1(\R^N)$.
\end{Rem}

\section{Solutions to \eqref{eq} and related properties}\label{SecU}
Throughout this section, it is assumed that (g0)--(g4) hold. We consider the family of functions and positive numbers $(u_\varepsilon,\lambda_\varepsilon)_{\varepsilon\in(0,1)}$ given by Theorem \ref{th:pert_n}.

\begin{Lem}\label{lem:bdd}
The quantities $|\nabla u_\varepsilon|_2$ and $\lambda_\varepsilon$ are bounded.
\end{Lem}
\begin{proof}
From (g0), (g1), and (g3), for every $\delta>0$ there exists $C_\delta>0$ such that for every $s\in\R$
\[
G_+(s)\le C_\delta|s|^2+\delta|s|^{2+4/N}.
\]
Let us first consider the case $N \ge 3$. Since $\lambda_\varepsilon>0$, from \eqref{eq:Pohopert} we obtain
\[
\frac1{2^*} |\nabla u_\varepsilon|_2^2 < \int_{\R^N} G_+(u_\varepsilon) \, dx \le C_\delta |u_\varepsilon|_2^2 + \delta|u_\varepsilon|_{2+4/N}^{2+4/N} \le C_\delta \rho^2 + \delta C_{N,2+4/N}^{2+4/N} \rho^{4/N} |\nabla u_\varepsilon|_2^2.
\]
Taking $\delta < \left(2^* C_{N,2+4/N}^{2+4/N} \rho^{4/N}\right)^{-1}$, we obtain the boundedness of $|\nabla u_\varepsilon|_2$. In addition, this and again \eqref{eq:Pohopert} yield
\begin{equation}\label{eq:lambepsbdd}
0 < \frac{2^*}2 \lambda_\varepsilon \rho^2 < \int_{\R^N} |\nabla u_\varepsilon|^2 + \frac{2^*}2 \lambda_\varepsilon |u_\varepsilon|^2 + 2^* G_-^\varepsilon(u_\varepsilon) \, dx = 2^* \int_{\R^N} G_+(u_\varepsilon) \, dx \le C < \infty
\end{equation}
and $\lambda_\varepsilon$ is bounded as well. Now, let us consider the case $N = 2$. From \eqref{eq:Pohopert} we get
\begin{equation}\label{eq:lambepsbdd2}
\rho^2 \lambda_\varepsilon \le \int_{\R^N} \lambda_\varepsilon |u_\varepsilon|^2 + 2 G_-^\varepsilon(u_\varepsilon) \, dx = 2 \int_{\R^N} G_+(u_\varepsilon) \, dx \le C_\delta \rho^2 + \delta C_{N,2+4/N}^{2+4/N} \rho^{4/N} |\nabla u_\varepsilon|_2^2.
\end{equation}
Moreover, again using \eqref{eq:Pohopert} we obtain
\begin{equation*}
-\beta > J_\varepsilon(u_\varepsilon) = \frac12 \left(|\nabla u_\varepsilon|_2^2 - \rho^2 \lambda_\varepsilon\right),
\end{equation*}
which, together with \eqref{eq:lambepsbdd2}, implies
\[
\frac12 |\nabla u_\varepsilon|_2^2 \le C_\delta \rho^2 - \beta + \delta C_{N,2+4/N}^{2+4/N} \rho^{4/N} |\nabla u_\varepsilon|_2^2.
\]
Taking $\delta < \left(2 C_{N,2+4/N}^{2+4/N} \rho^{4/N}\right)^{-1}$, we obtain the boundedness of $|\nabla u_\varepsilon|_2$, hence that of $\lambda_\varepsilon$ follows from \eqref{eq:lambepsbdd2}.
\end{proof}

%If the case $u_\varepsilon=0$ in Theorem \ref{th:pert} can be ruled out even when $\eta>0$, then Lemma \ref{lem:bdd} work for $\eta>0$ because we can take $\rho$ sufficiently small, i.e., \eqref{eq:etainfsub} needs to be replaced with the stronger condition $2^*C_{N,2_\#}^{2_\#}\eta\rho^{4/N}<1$.

Note that for every $\varepsilon\in(0,1)$ and every $u\in\cD$ there holds $J_\varepsilon(u_\varepsilon)\le J_\varepsilon(u)\le J(u)$, whence $J_\varepsilon(u_\varepsilon)\le \inf_{\cD(\rho)} J =: c(\rho)$.
%Moreover, if $\varepsilon\le1/2$, then $J_\varepsilon(u_\varepsilon)\ge J_{1/2}(u_\varepsilon)\ge J_{1/2}(u_{1/2})$, which in turn yields, together with the Poho\v{z}aev identity and since $\lambda_\varepsilon>0$,
%\begin{equation}\label{eq:G+bddaw}
%J_{1/2}(u_{1/2})\le J_\varepsilon(u_\varepsilon)<\frac2{N-2}\int_{\R^N}G^\varepsilon(u_\varepsilon)\,dx\le\frac2{N-2}\int_{\R^N}G_+(u_\varepsilon)\,dx.
%\end{equation}

\begin{proof}[Proof of Theorem \ref{th:mainsub}]
From Lemma \ref{lem:bdd} there exist $u \in \cD$ and $\lambda \ge 0$ such that, up to a subsequence, $u_\varepsilon \rightharpoonup u$ in $H^1(\R^N)$ and $\lambda_\varepsilon \to \lambda$ as $\varepsilon \to 0^+$. %From \eqref{eq:G+bddaw} and the generalized Lions lemma (for functions in $H^1(\R^N)$), cf. \cite[Lemma 1.6]{MederskiBL2}, $u\ne0$.
Arguing as in \cite[Proof of Theorem 1.1]{MederskiBL2} we obtain that $\varphi_\varepsilon(u_\varepsilon)g_-(u_\varepsilon)v\to g_-(u)v$ a.e. for every $v\in\cC_0^\infty(\R^N)$ and that $g_+(u_\varepsilon)v$ and $\varphi_\varepsilon(u_\varepsilon)g_-(u_\varepsilon)v$ are uniformly integrable (and tight).
We deduce that for every $v\in\cC_0^\infty(\R^N)$
\[
-\lambda\int_{\R^N}uv\,dx\gets J_\varepsilon'(u_\varepsilon)v\to J'(u)v,
\]
i.e.,
\[
-\Delta u+\lambda u=g(u) \quad \text{in } \R^N.
\]
Moreover, \eqref{eq:lambepsbdd} (when $N \ge 3$) or \eqref{eq:lambepsbdd2} (when $N = 2$) yields
\[
\int_{\R^N}G_-^\varepsilon(u_\varepsilon)\,dx\le\int_{\R^N}G_+(u_\varepsilon)\,dx\le C
\]
hence, in view of Fatou's lemma, $G_-(u)\in L^1(\R^N)$. In particular, from \cite[Proposition 3.1]{MederskiBL2}, $(u,\lambda)$ satisfies the Poho\v{z}aev identity
\begin{equation}\label{eq:PohoNoEps}
(N-2) \int_{\R^N} |\nabla u|^2 \, dx = 2N \int_{\R^N} G(u) - \frac12 \lambda |u|^2 \, dx.
\end{equation}
We can assume that $|\nabla u_\varepsilon|_2$ and $\int_{\R^N}G_-^\varepsilon(u_\varepsilon)\,dx$ are convergent. Since $\int_{\R^N}G_+(u_\varepsilon)\,dx\to\int_{\R^N}G_+(u)$ (remember that each $u_\eps$ is radially symmetric), we have $J(u)\le\lim_\varepsilon J_\varepsilon(u_\varepsilon)<0$ from Lemma \ref{lem:neg}; in particular, $u\ne0$. If $\lambda=0$, then from \eqref{eq:PohoNoEps} we have $J(u)=|\nabla u|_2^2/N>0$, therefore $\lambda>0$. We prove that $u_\varepsilon\to u$ in $H^1(\R^N)$. Since $\lambda>0$, from \eqref{eq:PohoNoEps} there follows
\[\begin{split}
\int_{\R^N} (N-2) |\nabla u|^2 + N \lambda|u|^2 + 2N G_-(u) \, dx & \le \lim_\varepsilon \int_{\R^N} (N-2) |\nabla u_\varepsilon|^2 + N \lambda|u_\varepsilon|^2 + 2N G^\varepsilon_-(u_\varepsilon) \, dx\\
& = 2N \lim_\varepsilon \int_{\R^N} G_+(u_\varepsilon) \, dx = 2N \int_{\R^N} G_+(u) \, dx\\
& = \int_{\R^N}(N-2) |\nabla u|^2 + N \lambda|u|^2 + 2N G_-(u) \, dx,
\end{split}\]
which yields $|u_\varepsilon|_2 \to |u|_2$;
%and $\int_{\R^N} G_-^\varepsilon(u_\varepsilon) \, dx \to \int_{\R^N} G_-(u) \, dx$
in particular, $u\in\cS$. Furthermore, in a similar fashion we obtain $J(u) \le \lim_\varepsilon J_\varepsilon(u_\varepsilon) \le c(\rho) \le J(u)$, which shows that $|\nabla u_\varepsilon|_2 \to |\nabla u|_2$ and $J(u) = c(\rho)$. Since \eqref{eq:pert} is translation-invariant, we can assume that every $u_\eps$ is radial, hence so is $u$. Moreover, because $u \ne 0$, there exists $x \in \R^N$ such that $u(x) \ne 0$ and so $u_\varepsilon(x)$ has the same sign as $u(x)$ for every $\varepsilon$ sufficiently small; since every $u_\varepsilon$ has constant sign, it has everywhere the same sign as $u(x)$, hence $u$ has constant sign too. Finally, if there exist $x,y,z \in \R^N$ such that $|x| < |y| < |z|$ and either $u(y) < \min\{u(x),u(z)\}$ or $u(y) > \max\{u(x),u(z)\}$, then arguing as before we obtain a contradiction.
\end{proof}

\begin{Rem}\label{rem:conv}
(i) The Proof of Theorem \ref{th:mainsub} contains the relevant result that $c_\eps(\rho) \to c(\rho)$ as $\eps \to 0^+$.\\
(ii) Unlike the proof of Theorem \ref{th:pert_n}, we cannot use the information $\lambda > 0$ to deduce $u \in \cS$ because we do not know whether $u$ is a critical point of $J|_\cD$.
\end{Rem}

\begin{proof}[Proof of Proposition \ref{pr:gsem}]
It follows from Lemma \ref{lem:neg} and Remark \ref{rem:conv} (i).
\end{proof}

%Let $2<p\le2^*$ ($p>2$ if $N=2$), $m\in\R$, $\mu<0$, and $\alpha>0$ {\color{red} ($m=0$?)}. Consider
%\begin{equation}\label{eq:nonexG}
%G_m(t):=\frac12\left(\alpha\log|t|-\frac\alpha2+m\right)t^2+\frac\mu{p}|t|^p.
%\end{equation}
%We consider nonexistence results for
%\begin{equation}\label{eq:nonex}
%-\Delta u+\lambda u=G_m'(u), \quad (\lambda,u)\in\R\times H^1(\R^N).
%\end{equation}

\begin{proof}[Proof of Theorem \ref{th:relcom}]
Let $u_n$ and $\rho_n$ as in the statement. We prove preliminarily that $u_n$ is relatively compact up to translations in $L^p(\R^N)$, $2 < p < 2^*$. Fix $\eps>0$. Arguing (by contradiction) as in the proof of Lemma \ref{lem:relcom}, we find $\beta,\gamma>0$, $\delta \ge 0$, $u \in \cD(\beta) \setminus \{0\}$, $v \in \cD(\gamma) \setminus \{0\}$, $y_n,z_n \in \R^N$, and $v_n,w_n \in H^1(\R^N)$ bounded, none of which depending on $\eps$, such that $\rho^2 - \beta^2 - \gamma^2 \ge \delta^2$ and $\lim_n J_\eps(u_n) - J_\eps(w_n) = J_\eps(u) + J_\eps(v)$, where $v_n(x) = u_n(x-y_n) - u(x)$ and $w_n(x) = v_n(x-z_n) - v(x)$. If $\delta>0$, then $\tilde{w}_n := \frac{\delta}{|w_n|_2}w_n \in \cS(\delta)$ and, again from \cite[Lemma 2.4]{Shibata_2014},
\begin{equation*}
\begin{split}
c_\eps(\rho) & = \lim_n J_\eps(u_n) = J_\eps(u) + J_\eps(v) + \lim_n J_\eps(\tilde{w}_n) \ge c_\eps(\beta) + c_\eps(\gamma) + c_\eps(\delta)\\
& \ge c_\eps(\sqrt{\beta^2 + \gamma^2 + \delta^2}) \ge c_\eps(\rho).
\end{split}
\end{equation*}
Letting $\eps \to 0^+$, observing that $J_\eps(u) \to J(u)$ and $J_\eps(v) \to J(v)$ from the monotone convergence theorem, and using Remark \ref{rem:conv} (i), we obtain
\[\begin{split}
c(\rho) & \ge J(u) + J(v) + \lim_{\eps \to 0^+} \lim_n J_\eps(\tilde{w}_n) \ge \lim_{\eps \to 0^+} c_\eps(\beta) + \lim_{\eps \to 0^+} c_\eps(\gamma) + \lim_{\eps \to 0^+} c_\eps(\delta)\\
& \ge \lim_{\eps \to 0^+} c_\eps(\sqrt{\beta^2 + \gamma^2 + \delta^2}) \ge c(\rho).
\end{split}\]
In particular, $c_\eps(\beta) \to J(u)$ and $c_\eps(\gamma) \to J(v)$ as $\eps \to 0^+$, therefore
\[
c(\beta) \le J(u) = \lim_{\eps \to 0^+} c_\eps(\beta) \le c(\beta), \quad c(\gamma) \le J(v) = \lim_{\eps \to 0^+} c_\eps(\gamma) \le c(\gamma),
\]
and so, from Lemma \ref{lem:subadd} and Remark \ref{rem:subadd} (a),
\[\begin{split}
\lim_{\eps \to 0^+} c_\eps(\sqrt{\beta^2 + \gamma^2 + \delta^2}) & \le \lim_{\eps \to 0^+} c_\eps(\sqrt{\beta^2 + \gamma^2}) + \lim_{\eps \to 0^+} c_\eps(\delta) \le c(\sqrt{\beta^2 + \gamma^2}) + \lim_{\eps \to 0^+} c_\eps(\delta)\\
& < c(\beta) + c(\gamma) + \lim_{\eps \to 0^+} c_\eps(\delta) = \lim_{\eps \to 0^+} c_\eps(\sqrt{\beta^2 + \gamma^2 + \delta^2}),
\end{split}\]
a contradiction. If $\delta=0$, then, again as in the proof of Lemma \ref{lem:relcom}, $\liminf_n J_\eps(w_n) \ge 0$ and we get
\[
c_\eps(\rho) = \lim_n J_\eps(u_n) = J_\eps(u) + J_\eps(v) + \lim_n J_\eps(w_n) \ge c_\eps(\beta) + c_\eps(\gamma) \ge c_\eps(\sqrt{\beta^2 + \gamma^2}) \ge c_\eps(\rho),
\]
thus we obtain a contradiction arguing in a similar (and simpler) way as before. Having proved that $u_n$ is relatively compact up to translations in $L^p(\R^N)$, $2 < p < 2^*$, now we prove that the same convergence holds in $H^1(\R^N)$. Recalling that $u_n$ is bounded, there exists $u \in \cD(\rho)$ such that, up to translations and subsequences, $u_n \weakto u$ in $H^1(\R^N)$ and $u_n \to u$ in $L^p(\R^N)$, $2 < p < 2^*$, whence
\[
c(\rho) \le J(u) \le \lim_n J(u_n) = c(\rho)
\]
and so $|\nabla u_n|_2^2 \to |\nabla u|_2^2$. Moreover, $J(u) = c(\rho) < 0$, which implies, in particular, that $G(u) \in L^1(\R^N)$. Assume by contradiction that $|u|_2 < \rho$; then, for every $v \in \cC_0^\infty(\R^N)$ and every $t \in \R$ sufficiently small, $|u + tv|_2 < \rho$. Next, we check that $J$ is differentiable at $u$ along every $v \in \cC_0^\infty(\R^N)$. Of course, this only needs to be checked for the functional $\int_{\R^N} G(\cdot) \, dx$. Let $t \ne 0$: from (g0) and (g2), $\frac1t [G(u+tv) - G(u)] = \int_0^1 g(u + stv) \, ds \, v$ and the family $\{\int_0^1 g(u + stv) \, ds \, v\}_t$ is uniformly integrable (and tight), thus the claim holds true. There follows that $J'(u)[v] = 0$ for every $v \in \cC_0^\infty(\R^N)$, i.e., $-\Delta u = g(u)$ in $\R^N$. Then, from \cite[Proposition 3.1]{MederskiBL2}, $u$ satisfies \eqref{eq:PohoNoEps} with $\lambda = 0$, whence $J(u) = |\nabla u|_2^2 / N > 0$, which is impossible.
\end{proof}

\begin{proof}[Proof of Theorem \ref{th:nonex}]
Since $G(s) = \alpha / 2 (\ln s^2 - 1) s^2 + \mu / p |s|^p$, it is clear that {\color{red}(g0)--(g3)} hold. We want to prove that {\color{red}(g4)} is satisfied if and only if $\mu > -\frac{\alpha p}{p - 2} e^{-p/2}$. If $\mu \ge 0$, then clearly {\color{red}(g4)} holds, so let us consider the case $\mu < 0$.
Let $s>0$. Of course, $G(s)>0$ if and only if
\[
\widetilde{G}(s) := \frac\alpha2 \left(\ln s^2 - 1\right) + \frac\mu{p} s^{p-2} > 0.
\]
Since
\[
\widetilde{G}'(s) = \frac\alpha{s} + \mu \frac{p-2}p s^{p-3},
\]
we have 
\[
\max\widetilde{G} = \frac\alpha{2} \left(\ln\left(\frac{\alpha p}{\mu(2-p)}\right)^{2/(p-2)} - 1\right) - \frac\alpha{p-2}
\]
and the claim holds because {\color{red}(g4)} is equivalent to $\max\widetilde{G}>0$. Point (i) is then a consequence of Theorem \ref{th:mainsub}. Concerning point (ii), observe that $\max G = 0$ if and only if $\mu = - \frac{\alpha p}{p - 2} e^{-p/2}$. Since, from \cite[Proposition 3.1]{MederskiBL2}, every solution to \eqref{eq} satisfies the Poho\v{z}aev identity \eqref{eq:PohoNoEps}, the statement easily holds true if $\lambda > 0$ or $N \ge 3$ or $\mu < - \frac{\alpha p}{p - 2} e^{-p/2}$. Now let us consider the case $\lambda = 0$, $N = 2$, and $\mu = - \frac{\alpha p}{p - 2} e^{-p/2}$ and assume by contradiction that such a $u$ exists. Since \eqref{eq:PohoNoEps} reads $\int_{\R^N} G(u) \, dx = 0$, $G \le 0$, and $u$ is continuous (se, e.g., \cite[p. 271]{Struwe}), necessarily $u \equiv 0$ or $u \equiv \pm \bar{s}$, where $\bar{s}$ is the unique $s>0$ such that $G(s) = 0$, which contradicts $|u|_2 = \rho \in (0,\infty)$.
\end{proof}

\section{Multiple solutions}\label{SecM}
\subsection{Functional setting}
Let us recall from \cite{RaoRen} that a function $A \colon \R \to \R$ is called an \emph{\textit{N}-function} if and only if it is nonnegative, even, convex, and satisfies
\[
\lim_{s \to 0} \frac{A(s)}{s} = \lim_{s \to \infty} \frac{A(s)}{s} = \infty.
\]
It is said to satisfy the \emph{$\Delta_2$ condition} globally if and only if there exists $K>0$ such that
\begin{equation*}
A(2s) \le K A(s) \quad \text{for all } s \in \R.
\end{equation*}
It is said to satisfy the \emph{$\nabla_2$ condition} globally if and only if there exist $\ell > 1$ such that
\begin{equation*}
2\ell A(s) \le A(\ell s) \quad \text{for all } s \in \R.
\end{equation*}
Equivalently, $A$ satisfies the $\Delta_2$ (respectively, $\nabla_2$) condition globally if and only if there exists $C>1$ such that
\begin{equation}\label{eq:car}
C A(s) \ge s A'(s) \; \text{(respectively, $s A'(s) \ge C A(s)$)} \quad \text{for all } s \in \R.
\end{equation}
%\begin{equation}\label{eq:del-nab}
%C_2 A(s) \ge s A'(s) \ge C_1 A(s) \quad \text{for all } s \in \R.
%\end{equation}
From now on, $A$ shall denote the same function as in (A). We can define the \textit{Orlicz space} associated with $A$ as
\begin{equation*}
V := \left\{u \in L^1_\textup{loc}(\R^N) : A(u) \in L^1(\R^N)\right\}.
\end{equation*}
%Since $A$ satisfies the $\De_2$ contidion globally -- i.e., the inequality on the left in \eqref{eq:del-nab} -- $V$ is a vector space; in fact, i
If we define the norm
\begin{equation*}
\|u\|_V := \inf\left\{\kappa > 0 : \int_{\R^N} A(u/\kappa) \, dx \le 1\right\},
\end{equation*}
then $(V,\|\cdot\|_V)$ is a reflexive Banach space (cf. \cite[Theorem IV.I.10]{RaoRen}).
%It is reflexive because $A$ satisfies the $\De_2$ and $\nabla_2$ conditions globally.
Moreover, the following holds true.
\begin{Lem}\label{lem:recalls}
(i) Let $u_n,u \in V$. Then $\lim_n \|u_n - u\|_V = 0$ if and only if $\lim_n \int_{\R^N} A(u_n - u) \, dx = 0$.\\
(ii) Let $X \subset V$. Then $X$ is bounded if and only if
$$\left\{\int_{\R^N} A(u) \, dx : u \in X\right\}$$
is bounded.\\
(iii) Let $u_n,u \in V$. If $u_n \to u$ a.e. and $\int_{\R^N} A(u_n) \, dx \to \int_{\R^N} A(u) \, dx$, then $\|u_n - u\|_V \to 0$.
\end{Lem}
\begin{proof}
The first two points follow from \cite[Theorem III.IV.12 and Corollary III.IV.15]{RaoRen} respectively, while the third one is a consequence of the previous two and \cite[Theorem 2, Example (b)]{BrezisLieb}.
\end{proof}

Note that $W = H^1(\R^N) \cap V$. Since $(V,\|\cdot\|_V)$ is a reflexive Banach space, so is $(W,\|\cdot\|_W)$, with
\[
\|u\|_W^2 := |\nabla u|_2^2 + |u|_2^2 + \|u\|_V^2.
\]

We obtain the following variant of Lions' lemma.
\begin{Lem}\label{lem:Lions}
Suppose that $u_n \in W$ is bounded and for some $r>0$ 	
\begin{equation}\label{eq:LionsCond11}
\lim_{n\to\infty}\sup_{y\in \R^N} \int_{B(y,r)} |u_n|^2\,dx=0.
\end{equation}
Then $u_n\to 0$ in $L^p(\R^N)$  for every $p \in [2,2^*)$.
\end{Lem}
\begin{proof} In view of Lions' lemma \cite[Lemma I.1 in Part 2]{Lions84}, $u_n\to 0$ in $L^p(\R^N)$  for every $p \in (2,2^*)$. We prove the convergence in $L^2(\R^N)$.
Take any  $p \in (2,2^*)$ and $\eps>0$. We find $\delta>0$ such that 
\begin{equation*}\begin{split}
|s|^2 & \leq \eps A(s)\quad\hbox{ if }|s|\in [0,\delta],\\
%		\Psi(s)&\leq& c s^{p}\quad\hspace{1.5mm}\hbox{ for }s\in (\delta ,M],\\
|s|^2 & \leq \delta^{2-p} |s|^{p}\quad\hbox{ if }|s|>\delta,
\end{split}\end{equation*}
Hence, we get
$$\limsup_{n\to\infty}\int_{\R^N}|u_n|^2\, dx\leq \eps \limsup_{n\to\infty}\int_{\R^N}A(u_n)\, dx+\delta^{2-p} \limsup_{n\to\infty}\int_{\R^N}|u_n|^p\, dx=\eps \limsup_{n\to\infty}\int_{\R^N}A(u_n)\, dx.$$
In view of Lemma \ref{lem:recalls} (ii), we conclude by letting $\eps\to0$.
\end{proof}

Let $\cO$ be any subgroup $\cO(N)$ such that  $\R^N$ is compatible with $\cO$
(cf. \cite{Lions82}), i.e.,
$\lim_{|y|\to\infty} m(y,r)=\infty$ for some $r>0$,
where, for $y \in \R^N$,
$$m(y,r):=\sup\big\{n\in\N: \hbox{there exist }g_1,\dots,g_n\in\cO\hbox{ such that } B(g_iy,r)\cap B(g_jy,r)=\emptyset\hbox{ for }i\neq j\big\}.$$

In view of \cite{Lions82}, $H^1_{\cO}(\R^N)$ embeds compactly into $L^p(\R^N)$ for $2<p<2^*$. As for the subspace $W_\cO:=W\cap H^1_{\cO}(\R^N)$, we obtain the compact embedding also in $L^2(\R^N)$ .

\begin{Cor}\label{CorLions2}
Suppose that $u_n \in W_{\cO}$ is bounded and $u_n\to 0$ in $L^2_\textup{loc}(\R^N)$. Then  $u_n\to 0$ in $L^p(\R^N)$  for every $p \in [2,2^*)$.
\end{Cor}
\begin{proof}
Suppose that
\begin{equation}\label{eq:LionsCond12proof1}
\int_{B(y_n,1)} |u_n|^2\,dx\geq c>0
\end{equation}
for some sequence $y_n \in \R^N$ and a constant $c$.
Observe that in the family $\{B(hy_n,1)\}_{h\in\cO}$ we find an increasing number of disjoint balls provided that $|y_n|\to\infty$. Since $u_n$ is bounded in $L^2(\R^N)$ and invariant with respect to $\cO$, by \eqref{eq:LionsCond12proof1} $(y_n)$ must be bounded. Then for sufficiently large $r$ one obtains
$$\int_{B(0,r)} |u_n|^2\,dx\geq \int_{B(y_n,1)} |u_n|^2\,dx\geq  c>0,$$
and we get a contradiction with the convergence of $u_n$ in $L^2_\textup{loc}(\R^N)$. Therefore by Lemma \ref{lem:Lions} we conclude.
\end{proof}

%In addition, we have what follows.
%
%\begin{Lem}\label{Lem:cpt}
%The embedding $W_r \hookrightarrow L^p(\R^N)$ is compact for every $p \in [2,2^*)$.
%\end{Lem}
%\begin{proof}
%Since $\lim_{s \to 0} A(s) / s^2 = \infty$, this is a consequence of \cite[Lemmas 1 and 2]{Strauss} taking, for every $N \ge 1$, $P_N(s) = s^2$ and $Q_N(s) = A(s) + |s|^q$, where $q \in (p,2^*)$.
%\end{proof}

Finally, we have the following result.
\begin{Prop}\label{prop:variation}
If (A) and (f0)--(f3) hold, then the functional $J|_W \colon W \to \R$ is of class $\cC^1$ and, for every $u \in W \cap \cS$, $J|_{W \cap \cS}'(u) = 0$ if and only if there exists $\la \in \R$ such that $(u,\la)$ is a solution to \eqref{eq}--\eqref{eq:mass}.
\end{Prop}
\begin{proof}
Let $B \colon \R \to \R$ be the complementary \textit{N}-function of $A$. From \cite[Theorem II.V.8]{RaoRen}, it satisfies the $\De_2$ and $\nabla_2$ conditions globally because $A$ does. Let us recall (cf. \cite[Definition III.IV.2, Corollary III.IV.5, and Corollary IV.II.9]{RaoRen}) that $V'$, the dual space of $V$, is isomorphic to the Orlicz space associated with $B$. From \cite[Theorem I.III.3]{RaoRen},
\[
B\bigl(a(s)\bigr) = sa(s) - A(s) \le (C-1) A(s),
\]
where $C>1$ is the constant given in the characterization of the $\De_2$ condition -- i.e., \eqref{eq:car}. Then, for every $u,v \in V$, we have
\[
\left|\int_{\R^N} a(u) v \, dx\right| \le \int_{\R^N} |a(u)| |v| \, dx \lesssim \|a(u)\|_{V'} \|v\|_V.
\]
Now we can use the same argument as in \cite[Lemma 2.1]{CGMS} to obtain that the functionals $u \mapsto \int_{\R^N} A(u) \, dx$ and $u \mapsto \int_{\R^N} F(u) \, dx$ belong to $\cC^1(W)$. The remaining part is obvious.
\end{proof}

\subsection{Proof of Theorem \ref{th:multip}}
We recall the definitions $F_- := F_+ - F$ and $f_- := F_-'$.

\begin{Lem}\label{lem:coerbb2}
If (A), (f0)--(f3), and \eqref{eq:eta-rho} hold, then $J|_{W \cap \cS}$ is coercive and bounded below.
\end{Lem}
\begin{proof}
For every $\delta > 0$, there exists $C_\delta > 0$ such that for every $s \in \R$
\[
F_+(s) \le C_\delta |s|^2 + (\eta + \delta) |s|^{2+4/N}.
\]
For every $u \in W \cap \cS$ there holds
\[\begin{split}
J(u) & \ge \int_{\R^N} \frac12 |\nabla u|^2 + A(u) - F_+(u) \, dx\\
& \ge \int_{\R^N} \frac12 |\nabla u|^2 + A(u) - C_\delta |u|^2 - (\eta + \delta) |u|^{2+4/N} \, dx\\
& \ge \frac12 |\nabla u|_2^2 + \int_{\R^N} A(u) \, dx - C_\delta \rho^2 - (\eta + \delta) C_{N,2+4/N}^{2+4/N} \rho^{4/N} |\nabla u|_2^2.
\end{split}\]
We conclude by Lemma \ref{lem:recalls} (ii) and taking $\delta$ so small that $2 (\eta + \delta) C_{N,2+4/N}^{2+4/N} \rho^{4/N} < 1$.
\end{proof}

\begin{Rem}\label{rem:+-}
Setting $f_\textup{p} := \max\{f,0\}$ and $f_\textup{n} := \max\{-f,0\}$, we have that $f_-(s) = f_\textup{n}(s)$ if $s \ge 0$ and $f_-(s) = -f_\textup{p}(s)$ if $s < 0$. In particular, $f_-(s)s \ge 0$ for every $s \in \R$.
\end{Rem}

\begin{Lem}\label{Lem:PS}
If (A) and (f0)--(f3) hold, then $J|_{W_{\cO}} \cap \cS$ satisfies the Palais--Smale condition.
\end{Lem}
\begin{proof}
Let $u_n \in W_{\cO} \cap \cS$ such that $J(u_n)$ is bounded and $J|_{W_{\cO} \cap \cS}'(u_n) \to 0$. From Lemma \ref{lem:coerbb2}, $u_n$ is bounded in $W$, therefore there exists $u \in W_{\cO}$ such that $u_n \weakto u$ in $W_\cO$ up to a subsequence. Then, from Corollary \ref{CorLions2}, $u_n \to u$ in $L^p(\R^N)$ for every $p \in [2,2^*)$; in particular, $u \in \cS$. Up to a second subsequence, we can assume that $u_n \to u$ a.e. in $\R^N$. Additionally, from \cite[Lemma 3]{BerLionsII}, there exist $\lambda_n \in \R$ such that
\begin{equation}\label{eq:lambda}
-\Delta u_n + \lambda_n u_n - g(u_n) u_n \to 0 \quad \text{in } W_\cO',
\end{equation}
where $W_\cO'$ is the dual space of $W_\cO$. Testing \eqref{eq:lambda} with $u_n$, we obtain that $\lambda_n$ is bounded as well, hence there exists $\lambda \in \R$ such that $\lambda_n \to \la$ up to a subsequence, and $(u,\la)$ is a solution to \eqref{eq}.
%From Fatou's lemma, the weak convergence in $H^1(\R^N)$, (f1), and (f4), we have
%\[\begin{split}
%J(u) & = \frac12 \int_{\R^N} |\nabla u|^2 \, dx + \int_{\R^N} A(u) \, dx + \int_{\R^N} F_-(u) \, dx - \int_{\R^N} F_+(u) \, dx\\
%& \le \lim_n \frac12 \int_{\R^N} |\nabla u_n|^2 \, dx + \int_{\R^N} A(u_n) \, dx + \int_{\R^N} F_-(u_n) \, dx - \int_{\R^N} F_+(u_n) \, dx = \lim_n J(u_n) \le 0.
%\end{split}\]
%If $\lambda \le 0$, then from the Poho\v{z}aev identity we obtain $0 \ge J(u) \ge |\nabla u|_2^2/N > 0$, thus $\lambda > 0$.
Finally, from the Nehari identity and the fact that $\la_n \to \la$, we obtain
\[\begin{split}
\int_{\R^N} |\nabla u|^2 + a(u)u + f_-(u)u \, dx & = \int_{\R^N} f_+(u)u - \lambda |u|^2 \, dx = \lim_n \int_{\R^N} f_+(u_n)u_n - \lambda |u_n|^2 \, dx\\
& = \lim_n \int_{\R^N} |\nabla u_n|^2 + a(u_n)u_n + f_-(u_n)u_n \, dx,
\end{split}\]
which, together with Remark \ref{rem:+-}, implies that $|\nabla u_n|_2^2 \to |\nabla u|_2^2$ and $\int_{\R^N} a(u_n)u_n \, dx \to \int_{\R^N} a(u)u \, dx$. It remains to prove that $\|u_n - u\|_V \to 0$. In virtue of Lemma \ref{lem:recalls}(iii), it suffices to prove that $\int_{\R^N} A(u_n) \, dx \to \int_{\R^N} A(u) \, dx$. Additionally, since $A$ satisfies the $\nabla_2$ condition globally, this occurs if the sequence $a(u_n)u_n$ is bounded above by an integrable function, which holds true from (A) and \cite[Example (b)]{BrezisLieb}.
\end{proof}

\begin{Lem}\label{lem:BLmap}
For every $k \ge 1$ there exist $\pi_k \colon \mathbb{S}^{k-1} \to W_{\cO(N)} \cap \cS$ and $\tilde \pi_k \colon \mathbb{S}^{k-1} \to \cX \cap \cS$ odd and continuous.
\end{Lem}
\begin{proof}
The existence of $\pi_k$ was proved in \cite[Theorem 10, Lemma 8]{BerLionsII} (with the notations therein, note that we can take any number different from $0$ instead of $\zeta$ because we do not need $V(u) \ge 1$). Then, following \cite[Proof of Lemma 3.4]{JeanjeanLu:norm} and \cite[Remark 4.2]{MederskiBL}, we take $\phi \in \cC^\infty(\R)$ odd such that $0 \le \phi \le 1$ and $\phi(t) = 1$ for all $t \ge 1$, and define for $\sigma \in \mathbb{S}^{k-1}$
\[
\tilde{\pi}_k(\sigma)(x) := \pi_k(\sigma)(x) \phi(|x_1| - |x_2|) \quad \text{for all } x = (x_1,x_2,x_3) \in \R^M \times \R^M \times \R^{N-2M}.\qedhere
\]
\end{proof}

We make use of the following abstract theorem, where $\cG$ stands for the Krasnoselsky genus \cite[Chapter 5]{Struwe}.
\begin{Th}\label{Th:JL}
Let $E$ be a Banach space, $H \supset E$ a Hilbert space with scalar product $(\cdot|\cdot)$, $R>0$, $\cM := \left\{u \in E : (u|u) = R\right\}$, and $I \in \cC^1(E)$ even such that $I|_\cM$ is bounded below. For every $k \ge 1$, define
\[
c_k := \inf_{A \in \Gamma_k} \sup_{u \in A} I(u), \quad \Gamma_k := \left\{A \subset \cM : A = -A = \overline{A} \text{ and } \cG(A) \ge k\right\}.
\]
Then, for every $k \ge 1$, $-\infty < c_1 \le \dots \le c_k \le c_{k+1} \le \dots$ and the following holds: if there exist $k \ge 1$ and $h \ge 0$ such that $c_k = \dots = c_{k+h} < \infty$ and $I|_\cM$ satisfies the Palais--Smale condition at the level $c_k$, then
\[
\cG\left(\left\{u \in \cM : I(u) = c_k \text{ and } I|_\cM'(u) = 0\right\}\right) \ge h + 1
\]
(in particular, taking $h=0$, every $c_k$ is a critical value of $I|_\cM$).
\end{Th}
\begin{proof}
This theorem is basically (part of) \cite[Theorem 2.1]{JeanjeanLu:norm} (see also \cite{Rabinowitz:1986}), so we omit the proof. The only difference is that the values $c_k$ are critical regardless of their sign, which is a consequence of $I|_\cM$ satisfying the Palais--Smale condition at any level.
\end{proof}

\begin{proof}[Proof of Theorem \ref{th:multip}]
Let us set $E = W_{\cO(N)}$ (respectively, $E = \cX$), $H = L^2(\R^N)$, $R = \rho^2$, $\cM = \cS \cap E$, and $I = J|_E$. From Lemmas \ref{lem:coerbb2} and \ref{Lem:PS}, $I|_\cM = J_{\cS \cap E}$ is bounded below and satisfies the Palais--Smale condition; moreover, from Lemma \ref{lem:BLmap}, $\pi_k(\mathbb{S}^{k-1}) \in \Gamma_k$ (respectively, $\tilde \pi_k(\mathbb{S}^{k-1}) \in \Gamma_k$) for every $k$, so the numbers $c_k$ are finite. Applying Theorem \ref{Th:JL}, we conclude the part about the existence of infinitely many solutions. Concerning the existence of a least-energy solution, we consider a sequence $u_n \in \cS \cap E$ such that $\lim_n J(u_n) = \inf_{\cS \cap E} J$. From Ekeland's variational principle, we can assume that $u_n$ is a Palais--Smale sequence for $J|_{\cS \cap E}$, hence argue as above to obtain a solution $(\bar{u},\bar{\lambda}) \in \R \times (\cS \cap E)$ to \eqref{eq}--\eqref{eq:mass} such that $J(\bar{u}) = \min_{\cS \cap E} J$. The fact that $\min_{\cS \cap W_{\cO(N)}} J = \min_{\cS \cap W} J$ follows from the properties of the Schwartz rearrangement \cite[Chapter 3]{LiebLoss}.
\end{proof}

\noindent{\bf Data availability statement.}
Data sharing not applicable to this article as no datasets were generated or analysed during the current study.

\vspace{.5\baselineskip}

\noindent{\bf Acknowledgements.}
The authors would like to thank the referee for valuable comments concerning the improvement of Theorem 1.6 and for indicating the very recent preprint \cite{ZhangZhang}, where properties of the ground state energy have been studied in the context of strongly sublinear nonlinearities.
%\subsection*{Acknowledgements}
J.M. was %partially
partly supported by the National Science Centre, Poland (Grant No. 2017/26/E/ST1/00817). J.S. was %partially
partly supported by the National Science Centre, Poland (Grant No. 2020/37/N/ST1/00795).\\
J.S. is a member of GNAMPA.

\end{document}